\documentclass[12pt]{article}
\usepackage{amssymb,amsmath,amsthm,latexsym,graphicx}
\usepackage[english]{babel}
\usepackage{mathabx}
\newtheorem{thm}{Theorem}[section]

\newtheorem{lem}[thm]{Lemma}
\newtheorem{cor}[thm]{Corollary}

\newtheorem{defi}{Definition}

\newcommand{\comment}[1]{} 

\begin{document}
\title{Recent results on Choi's orthogonal Latin squares}

\author{Jon-Lark Kim \\ Department of Mathematics \\ Sogang University  \\ Seoul, 121-742, South Korea \\
{\tt jlkim@sogang.ac.kr } \thanks{Corresponding authors. J.-L. Kim was supported by Basic Research Program through the National Research Foundation of Korea (NRF) funded by the Ministry of Education (NRF-2019R1I1A1A01057755).} \\ \\ Dong Eun Ohk \\ Department of Mathematics \\ Sogang University  \\ Seoul, 121-742, South Korea \\
  \\ \\ Doo Young Park \\ Department of Mathematics \\ Sogang University  \\ Seoul, 121-742, South Korea \\  \\ \\ Jae Woo Park \\ Department of Mathematics \\ Sogang University  \\ Seoul, 121-742, South Korea \\
}

\date{09/08/2021}

\maketitle

\begin{abstract}
\noindent Choi Seok-Jeong studied Latin squares at least 60 years earlier than Euler although this was less known. He introduced a pair of orthogonal Latin squares of order 9 in his book. Interestingly, his two orthogonal non-double-diagonal Latin squares produce a magic square of order 9, whose theoretical reason was not studied. There have been a few studies on Choi's Latin squares of order 9. The most recent one is Ko-Wei Lih's construction of Choi's Latin squares of order 9 based on the two $3 \times 3$ orthogonal Latin squares.
 In this paper, we give a new generalization of Choi's orthogonal Latin squares of order 9 to orthogonal Latin squares of size $n^2$ using the Kronecker product including Lih's construction. We find a geometric description of Choi's orthogonal Latin squares of order 9
  using the dihedral group $D_8$. We also give a new way to construct magic squares from two orthogonal non-double-diagonal Latin squares, which explains why Choi's Latin squares produce a magic square of order 9.
  \end{abstract}

{\bf{Key Words:}} Choi Seok-Jeong, Koo-Soo-Ryak, Latin squares, magic squares

{\bf AMS subject classification}: 05B15, 05B20

\section{Introduction}

A \textit{Latin square} of order $n$ is an $n\times n$ array in which $n$ distinct symbols are arranged so that each symbol occurs once in each row and column. This Latin square is one of the most interesting mathematical objects. It can be applied to a lot of branches of discrete mathematics including finite geometry, coding theory and cryptography~\cite{a3},~\cite{a1}. In particular, orthogonal Latin squares have been one of the main topics in Latin squares. The superimposed pair of two orthogonal Latin squares is also called a \textit{Graeco-Latin sqaure} by Leonhard Euler (1707-1783) in 1776~\cite{Euler}. It is known that the study of Latin squares was researched by Euler in the 18th century. However the Korean mathematician, Choi Seok-Jeong [Choi is a family name] (1646-1715) already studied Latin squares at least 60 years before Euler's work. A pair of two orthogonal Latin squares of order 9 was introduced in Koo-Soo-Ryak (or Gusuryak) written by Choi Seok-Jeong. The Koo-Soo-Ryak was listed as the first literature on Latin squares in the Handbook of Combinatorial Designs~\cite{ColDinitz}.


Let K be the matrix form of the superimposed Latin square of order 9 from Koo-Soo-Ryak:

$$K \; = \;
\begin{matrix}
(5,1) & (6,3) & (4,2) & (8,7) & (9,9) & (7,8) & (2,4) & (3,6) & (1,5)\\
(4,3) & (5,2) & (6,1) & (7,9) & (8,8) & (9,7) & (1,6) & (2,5) & (3,4)\\
(6,2) & (4,1) & (5,3) & (9,8) & (7,7) & (8,9) & (3,5) & (1,4) & (2,6)\\
(2,7) & (3,9) & (1,8) & (5,4) & (6,6) & (4,5) & (8,1) & (9,3) & (7,2)\\
(1,9) & (2,8) & (3,7) & (4,6) & (5,5) & (6,4) & (7,3) & (8,2) & (9,1)\\
(3,8) & (1,7) & (2,9) & (6,5) & (4,4) & (5,6) & (9,2) & (7,1) & (8,3)\\
(8,4) & (9,6) & (7,5) & (2,1) & (3,3) & (1,2) & (5,7) & (6,9) & (4,8)\\
(7,6) & (8,5) & (9,4) & (1,3) & (2,2) & (3,1) & (4,9) & (5,8) & (6,7)\\
(9,5) & (7,4) & (8,6) & (3,2) & (1,1) & (2,3) & (6,8) & (4,7) & (5,9)\\
\end{matrix}$$
Then we can separate $K$ into two Latin squares $L$ and $N$. To get a visible effect, let us color in each square.

\begin{figure}[h]
\centering
\includegraphics[scale=0.35]{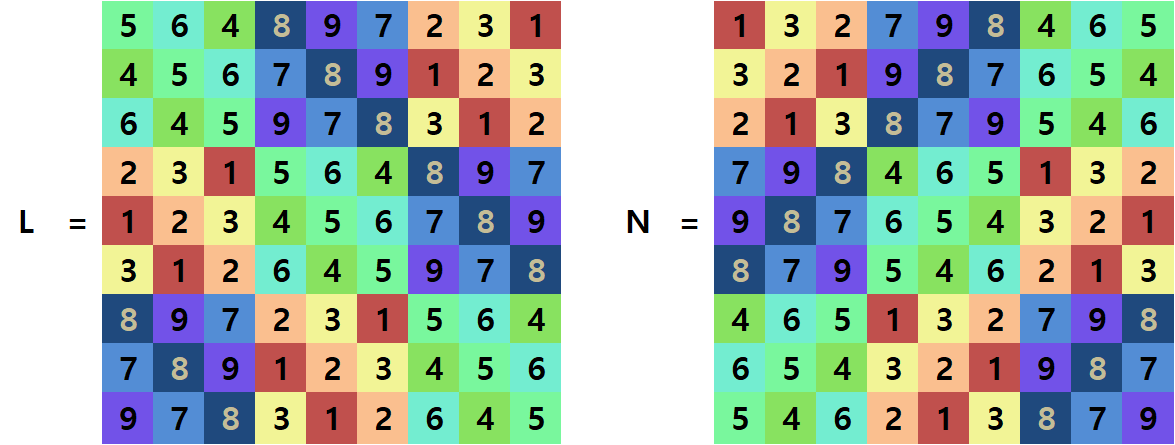}
\caption{A colored two Latin squares $L$ and $N$, respectively}
\end{figure}

We paint colors for each numbers, $1, 2, \cdots, 9$. In details, $1, 2, 3$ are colored in red, $4, 5, 6$ are colored in green, and $7, 8, 9$ are colored in blue. Then we observe that the Latin squares have self-repeating patterns. This simple structure of Choi's Latin squares motivates some generalization of his idea. We generalize Choi's Latin squares in three directions: the Kronecker product approach, the Dihedral group approach, and magic squares from Choi's Latin squares.

In this paper, we give a new generalization of Choi's orthogonal Latin squares of order 9 to orthogonal Latin squares of size $n^2$ using the Kronecker product including Lih's construction~\cite{a1}. There has been some attempt that the dihedral group $D_8$ acts on the Latin squares~\cite{a7}. We find a geometric description of Choi's orthogonal Latin squares of order 9 using $D_8$. We also give a new way to construct magic squares from two orthogonal non-double-diagonal Latin squares, which explains why Choi's Latin squares produce a magic square of order 9.
	
	


\section{A generalization of Choi's orthogonal Latin squares}

\begin{defi} \rm{(\cite{a1})}
Let $A=(a_{ij})$ be a Latin square of order $n (i,j\in \{1,2,\cdots,n\})$ and $B=(b_{st})$ be a Latin square of order $m (s,t\in \{1,2,\cdots,m\})$. Then the \textit{Kronecker product of $A$ and $B$}, which is an $mn\times mn$ square $A\otimes B$ given by
$$A\otimes B \, = \;
\begin{matrix}
(a_{11}, B) & (a_{12}, B) & \cdots & (a_{1n}, B)\\
(a_{21}, B) & (a_{22}, B) & \cdots & (a_{2n}, B)\\
\vdots & \vdots & \ddots & \vdots \\
(a_{n1}, B) & (a_{n2}, B) & \cdots & (a_{nn}, B)\\
\end{matrix}$$
where $(a_{ij},B)$ is the $m\times m$ square
$$(a_{ij},B) \, = \;
\begin{matrix}
(a_{ij}, b_{11}) & (a_{ij}, b_{12}) & \cdots & (a_{ij}, b_{1m})\\
(a_{ij}, b_{21}) & (a_{ij}, b_{22}) & \cdots & (a_{ij}, b_{2m})\\
\vdots & \vdots & \ddots & \vdots \\
(a_{ij}, b_{m1}) & (a_{ij}, b_{m2}) & \cdots & (a_{ij}, b_{mm})\\
\end{matrix}$$
\end{defi}

\begin{lem}\rm{(\cite{a1})}
$A\otimes B$ is a Latin square if $A$ and $B$ are both Latin squares.
\end{lem}

\begin{thm}\rm{(\cite{a1})}~\label{thm:knonecker}
If two Latin squares $A_{1}$ and $A_{2}$ of order $n$ are orthogonal and two Latin squares $B_{1}$ and $B_{2}$ of order $m$ are orthogonal, then $A_{1}\otimes B_{1}$ and $A_{2}\otimes B_{2}$ of order $mn$ are orthogonal.
\end{thm}

Now it is natural to substitute $m(a_{ij}-1) + b_{kl}$ for the entry $(a_{ij}, b_{kl})$ in $A\otimes B$. Thus we define the \textit{substituted Kronecker product $\otimes_S$ of two Latin squares $A$ and $B$} by the following block matrix
$$A\otimes_S B \, = \;
\begin{bmatrix}
(m(a_{11}-1)\times N_{m}+B) & \cdots & (m(a_{1n}-1)\times N_{m}+B)\\
\vdots & \ddots & \vdots \\
(m(a_{n1}-1)\times N_{m}+B) & \cdots & (m(a_{nn}-1)\times N_{m}+B)\\
\end{bmatrix}$$
where $A=(a_{ij})$ is a matrix of order $n$, $B$ is a matrix of order $m$, and $N_{m}$ is the $m\times m$ all-ones matrix.

Let us return to Latin squares. Judging from Figure 1, we can expect that $L$ is closely related to a Latin square of order 3. Let
$$A_3  = (a_{ij})=
\begin{matrix}
2 & 3 & 1\\
1 & 2 & 3\\
3 & 1 & 2\\
\end{matrix}$$

Then the following block matrix
$$\begin{bmatrix}
(3(a_{11}-1)\times N_{3}+A_3) & (3(a_{12}-1)\times N_{3}+A_3) & (3(a_{13}-1)\times N_{3}+A_3)\\
(3(a_{21}-1)\times N_{3}+A_3) & (3(a_{22}-1)\times N_{3}+A_3) & (3(a_{23}-1)\times N_{3}+A_3)\\
(3(a_{31}-1)\times N_{3}+A_3) & (3(a_{32}-1)\times N_{3}+A_3) & (3(a_{33}-1)\times N_{3}+A_3)\\	
\end{bmatrix}$$
produces $L$. In other words, $L=A_3\otimes_S A_3$. Similarly, let
$$B_3 \, = \;
\begin{matrix}
1 & 3 & 2\\
3 & 2 & 1\\
2 & 1 & 3\\
\end{matrix}$$
then $N=B_3\otimes_S B_3$. These two Latin squares $A_3$ and $B_3$ are elements of MOLS(3) which is the mutually orthogonal Latin squares of order 3. We recall that Lih~\cite{a2} also found this relation. However he did not explain why
$L=A_3\otimes_S A_3$ and $N=B_3\otimes_S B_3$ are orthogonal from the Kronecker product point of view.

\begin{cor}
Choi's two Latin squares of order 9 are orthogonal.
\end{cor}

\begin{proof}
By the above notation, we can put Choi's two Latin squares of order 9 by
 $L=A_3\otimes_S A_3$ and $N=B_3 \otimes_S B_3$. Note that $A_3$ and $B_3$ are orthogonal. Therefore, by taking $A_1=B_1=A_3$ and $A_2=B_2=B_3$ in Theorem~\ref{thm:knonecker} we see that
  $L=A_3\otimes_S A_3$ and $N=B_3\otimes_S B_3$ are also orthogonal.
\end{proof}

Hence it appears that Choi might know how to get the orthogonal Latin squares of order 9 by expanding orthogonal Latin squares of order 3.

It is natural to generalize Choi's approach to obtain orthogonal Latin squares by copying a smaller Latin square several times.

If $A$ is a Latin square of order $n$, we call $A \otimes_S A$ {\em Choi type Latin square of order $n^2$}.

Since there exists a pair of orthogonal Latin squares of order $n \ge 3$ and $n \ne 6$, the following is immediate.

\begin{cor}\label{cor-Choi's type}
There exists a pair of Choi's type Latin squares of order $n^2$ which are orthogonal whenever $n \ge 3$ and $n \ne 6$.
\end{cor}

We remark that Lih's construction~\cite{a2} gives only the case when $n=3$. Corollary~\ref{cor-Choi's type} extends this result to any $n \ge 3$ and $n \ne 6$.


\section{Latin squares acted by the dihedral group $D_8$}

We have noticed that $L$ is symmetric to $N$ with respect to the 5th column of $L$. In other word, if we let $L=(l_{ij})$, then $N=(l_{i(n+1-j)})$. So we define some operation.

\begin{defi}{\em
Let $A=(a_{ij})$ be an $n\times n$ matrix (or square or array). Define the $n\times n$ matrix $s_{2}(A)$ by $s_{2}(A)=(a_{i(n+1-j)})$.
}
\end{defi}

We can consider more symmetries. The \textit{dihedral group of degree $n$} denoted by $D_{2n}$ is a well-known group of order $2n$ consisting of symmetries on a regular $n$-polygon consisting rotations and reflections. In this case, we concentrate on a square, so the dihedral group of order 8, denoted by $D_{8}$, is needed. In $D_{8}$, there are eight elements, $s_{0},s_{1},s_{2},s_{3},r_{1},r_{2},r_{3},r_{4}$. Note that $s_{i}$ for $i=0, 1, 2, 3$ denotes a reflection. More precisely, $s_0$ is a horizontal reflection, $s_1$ is a main diagonal reflection, $s_2$ is a vertical reflection, and $s_3$ is an antidiagonal reflection. Note that
$r_0$ denotes the rigid motion and $r_i$'s $(i=1,2,3)$ denote  counterclockwise rotations by 90, 180, 270 degrees respectively so that  $r_2 = r_1^2$ and $r_3=r_1^3$.

We can define a set $D_8(A)=\{A, r_{1}(A), r_{2}(A), r_{3}(A), s_{0}(A), s_{1}(A), s_{2}(A), s_{3}(A)\}$ for a given Latin square $A$.

\begin{defi}\label{def:sigma}{\em
Let $L_{n}$ be the set of all Latin squares of order $n$. Then $\sigma\in D_{8}$ is a function with $\sigma\ : L_{n}\rightarrow L_{n}$ defined by
\begin{center}
$r_{0}(A)=(a_{ij})$, \; $r_{1}(A)=(a_{j(n+1-i)})$, \; $r_{2}(A)=(a_{(n+1-i)(n+1-j)})$, \\
$r_{3}(A)=(a_{(n+1-j)i})$, \; $s_{0}(A)=(a_{(n+1-i)j})$, \; $s_{1}(A)=(a_{ji})$ , \\
$s_{2}(A)=(a_{i(n+1-j)})$, \; $s_{3}(A)=(a_{(n+1-j)(n+1-i)})$
\end{center}
where $A \in L_n$ and $A=(a_{ij})$.
}
\end{defi}

Then we can regard an element in $D_{8}$ as a function acting on $L_n$. In fact, the dihedral group $D_8$ acts on $L_n$ (or $L_n$ is a $D_8$-set) as follows.
	
\begin{lem}
$L_n$ is a $D_8$-set.
\end{lem}

\begin{proof}
Let $\sigma \in D_8$ and $A \in L_n$. Since $A=(a_{ij})$ is a Latin square of order $n$, $\left\{ a_{1j}, a_{2j}, \cdots, a_{nj} \right\}$ $=$ $\left\{ a_{i1}, a_{i2}, \cdots, a_{in} \right\}$ $=$ $\left\{ 1, 2, \cdots, n \right\}$ for all $i,j = 1, 2, \cdots, n$. Thus by definition, $\sigma(A)$ is a Latin square.

If $\sigma = r_0$, then $r_0(A)=A$ for any $A \in L_n$.
Suppose that $\sigma_1, \sigma_2 \in D_8$. Let $\sigma_3=\sigma_1 \circ \sigma_2 \in D_8$. It is straightforward to check that
$\sigma_3(A)=\sigma_1(\sigma_2(A))$ by Definition~\ref{def:sigma}.
\end{proof}

In the Choi's Latin squares, $N = s_{2}(L)$ (or $L = s_{2}(N)$). Since $L$ and $N$ are orthogonal, we can say that $L$ and $s_{2}(L)$ are orthogonal. Then we can have some questions. Is $L$ orthogonal to $\sigma(L)$ for another $\sigma$ in $D_8$? And how many mutually orthogonal Latin squares are in the set $D_8(L)$? Moreover, for any Latin square $A$, what is the maximum number of mutually orthogonal Latin squares in the set $D_8(A)$?

\begin{lem}\label{lem:sigma_orth}
Suppose $A$ and $B$ are Latin squares of order $n$ and take an arbitrary $\sigma \in D_8$. Then $A$ is orthogonal to $B$ if and only if $\sigma(A)$ is orthogonal to $\sigma(B)$.
\end{lem}

By Lemma~\ref{lem:sigma_orth}, we have a criteria when two Latin squares in the set $D_8(A)$ are orthogonal. If two Latin squares $A$ and $B$ are orthogonal, we denote it by $A\bot B$:

\begin{center}
$r_0(A)\bot s_0(A)$ $\iff$ $r_1(A)\bot s_1(A)$ $\iff$ $r_2(A)\bot s_2(A)$ $\iff$ $r_3(A)\bot s_3(A)$\\
$r_0(A)\bot s_1(A)$ $\iff$ $r_1(A)\bot s_2(A)$ $\iff$ $r_2(A)\bot s_3(A)$ $\iff$ $r_3(A)\bot s_0(A)$\\
$r_0(A)\bot s_2(A)$ $\iff$ $r_1(A)\bot s_3(A)$ $\iff$ $r_2(A)\bot s_0(A)$ $\iff$ $r_3(A)\bot s_1(A)$\\
$r_0(A)\bot s_3(A)$ $\iff$ $r_1(A)\bot s_0(A)$ $\iff$ $r_2(A)\bot s_1(A)$ $\iff$ $r_3(A)\bot s_2(A)$\\
$r_0(A)\bot r_1(A)$ $\iff$ $r_1(A)\bot r_2(A)$ $\iff$ $r_2(A)\bot r_3(A)$\\
$\iff$ $r_3(A)\bot r_0(A)$ $\iff$ $s_0(A)\bot s_1(A)$ $\iff$\\
$s_1(A)\bot s_2(A)$ $\iff$ $s_2(A)\bot s_3(A)$ $\iff$ $s_3(A)\bot s_0(A)$\\
$r_0(A)\bot r_2(A)$ $\iff$ $r_1(A)\bot r_3(A)$ $\iff$ $s_0(A)\bot s_2(A)$ $\iff$ $s_1(A)\bot s_3(A)$\\
\end{center}

Thus for finding mutually orthogonal Latin squares in $D_8(A)$, we should look at the orthogonality of $A=r_0(A)$ and $\sigma(A)$ for $\sigma \in D_8$.

\begin{lem}\label{lem1}
For any $A \in L_n$, $A$ is not orthogonal to $r_2(A)$.
\end{lem}

\begin{proof}
Let $A=(a_{ij})$ and $r_2(A)=(b_{ij})$. Suppose that $A$ and $r_2(A)$ are orthogonal. Then we have

$$\left\{ (a_{ij},b_{ij})\, |\, i,j = 1, 2, \cdots, n \right\} = \left\{ (x,y)\, |\, x,y = 1, 2, \cdots, n \right\}.$$

Therefore there exist some integers $s_k,t_k$ such that $(a_{s_kt_k},b_{s_kt_k})=(k,k)$ for each nonnegative integer $k =1,2,\ldots n$. Let $n+1-s_k=s_k'$ and $n+1-t_k=t_k'$. Since $b_{s_kt_k} = a_{s_k't_k'}$ and $b_{s_k't_k'} = a_{s_kt_k}$, so $(a_{s_kt_k}, b_{s_kt_k}) = (b_{s_kt_k}, a_{s_kt_k}) = (a_{s_k't_k'}, b_{s_k't_k'})$. It means that the two ordered pairs $(a_{s_kt_k}, b_{s_kt_k})$ and $(a_{s_k't_k'}, b_{s_k't_k'})$ are the same in the set $\{ (a_{ij},b_{ij}) \}$. Since $A$ and $r_2(A)$ are orthogonal, we have $(s_k,t_k)=(s_k', t_k')$. That is, $s_k=s_k'$ and $t_k=t_k'$. This implies that $n=2s_k-1=2t_k-1$, that is, $s_k=t_k$ for any $k$. It contradicts.
\end{proof}

\begin{lem}\label{lem2}
Let $A \in L_n$ and $n$ be even. Then $A$ is not orthogonal to either $s_0(A)$ or $s_2(A)$.
\end{lem}

\begin{proof}
Suppose that $A$ is orthogonal to $s_0(A)$. Let $s_0(A)=(a_{(n+1-i)j})=b_{ij}$. By the similar argument of proof of Lemma~\ref{lem1}, there exist integer $u$ and $v$ such that $a_{uv}=b_{uv}=k$ for some $k$. So $a_{uv} =b_{uv} = a_{(n+1-u)v}$. Since $A$ is a Latin square, the entries in the $v$-th column are all distinct. Thus $a_{uv} = a_{(n+1-u)v}$ implies $u = n+1-u$ and so $u= (n+1)/2$. However, $n$ is even so that $u$ is not an integer. It contradicts. Hence $A$ is not orthogonal to $s_0(A)$. We can show that $A$ is not orthogonal to $s_2(A)$ in a similar manner.
\end{proof}

\begin{thm}
Let $A \in L_n$ and $n$ be odd. Then the maximum number of mutually orthogonal Latin squares of order $n$ in the set $D_8(A)$ is less than or equal to 4.

And if we assume that $n$ is even, then the maximum number of mutually orthogonal Latin squares in the set $D_8(A)$ is 2.
\end{thm}

\begin{proof} Let $M$ be the set of mutually orthogonal Latin squares, which has the maximum number of mutually orthogonal Latin squares in the set $D_8(A)$. By Lemma~\ref{lem1}, we can get $r_0(A)\notperp r_2(A)$, $r_1(A)\notperp r_3(A)$, $s_0(A)\notperp s_2(A)$ and $s_1(A)\notperp s_3(A)$. If we take three or more elements of $M$ from the set $\left\{ r_0(A), r_1(A), r_2(A), r_3(A) \right\}$, then there should appear a pair of non-orthogonal Latin squares. Similarly, we cannot take three or more elements from the set $\left\{ s_0(A), s_1(A), s_2(A), s_3(A) \right\}$. It means that the set $M$ can be $M = \{ r_{i_1}(A), r_{i_2}(A), s_{j_1}(A), s_{j_2}(A)\}$. Therefore we have that the maximum number of mutually orthogonal Latin squares in the set $D_8(A)$ is less than or equal to four.

Suppose $n$ is even and $M = \{ r_{i_1}(A), r_{i_2}(A), s_{j_1}(A), s_{j_2}(A)\}$. It is possible that $M$ does not contain $r_0(A)$, however, we can get the set of 4 mutually orthogonal Latin squares containing $r_0(A)$ by the group action. So without loss of generality, assume that $r_{i_1} = r_0$. By Lemma~\ref{lem2}, $s_{j_1}, s_{j_2}$ should be $1$ and $3$. However  $s_1(A)\notperp s_3(A)$ by Lemma~\ref{lem1}, so $| \, M \,| \ne 4$. Now suppose that $M = \{ r_0(A), r_{i_2}(A), s_{j_1}(A) \}$. Note that $i_2 = 1,3$ and $j_1 = 1,3$. However, Lemma~\ref{lem2} also implies that $r_1(A)\notperp s_1(A)$, $r_1(A)\notperp s_3(A)$, $r_3(A)\notperp s_1(A)$ and $r_3(A)\notperp s_3(A)$. Thus $| \, M \,| \ne 3$. Therefore $| \, M \,| = 2$ if $n$ is even.
\end{proof}

\begin{cor}
Let $L$ be one of Choi's Latin squares of order 9. Then the maximum number of mutually orthogonal Latin squares in $D_8(L)$ is two.
\end{cor}

\begin{proof}
By Theorem 3.5 there are at most 4 mutually orthogonal Latin squares in $D_8(L)$. Without loss of generality, we may assume that $L$ is one of them. We first show that there are only two mutually orthogonal Latin squares among $L, r_1(L), r_2(L), r_3(L)$. By Lemma 3.3, $L$ is not orthogonal to $r_2(L)$. This also implies that $r_1(L)$ is not orthogonal to $r_3(L)$. On the other hand, we have checked by enumerating all ordered pairs that $L$ is orthogonal to both $r_1(L)$ and $r_3(L)$. Therefore we have only two cases $\{L, r_1(L) \}$ and $\{L, r_3(L) \}$ among rotations.

We can easily check that $L$ is orthogonal to both $s_0(L)$ and $s_2(L)$ while $L$ is neither orthogonal to $s_1(L)$ nor to $s_3(L)$ because the two diagonal reflections do not change the value of 5 in the main diagonal. However $s_0(L)$ cannot be orthogonal to $s_2(L)$ because they reduce to $L$ and $r_2(L)$ which are not orthogonal by Lemma 3.3.

Therefore we have the following four possibilities.

\begin{enumerate}
\item $\{L, r_1(L), s_0(L) \}$

\item $\{L, r_1(L), s_2(L) \}$

\item $\{L, r_3(L), s_0(L) \}$

\item $\{L, r_3(L), s_2(L) \}$
\end{enumerate}

\[
r_1(L)=\begin{array}{c}
1 3 2 7 9 8 4 6 5 \\
3 2 1 9 8 7 6 5 4\\
2 1 3 8 7 9 5 4 6\\
7 9 8 {\bf 4} 6 5 1 3 2\\
9 8 7 6 5 {\bf 4} 3 2 1\\
8 7 9 5 4 6 2 1 3\\
4 6 5 1 3 2 7 9 8\\
6 5 4 3 2 1 9 8 7\\
5 4 6 2 1 3 8 7 9\\
\end{array}
~~~
s_0(L)=\begin{array}{c}
9 7 8 3 1 2 6 4 5\\
7 8 9 1 2 3 4 5 6\\
8 9 7 2 3 1 5 6 4\\
3 1 2 {\bf 6} 4 5 9 7 8\\
1 2 3 4 5 {\bf 6} 7 8 9\\
2 3 1 5 6 4 8 9 7\\
6 4 5 9 7 8 3 1 2\\
4 5 6 7 8 9 1 2 3\\
5 6 4 8 9 7 2 3 1\\
\end{array}
~~~
s_2(L)=\begin{array}{c}
1 3 2 7 9 8 4 6 5\\
3 2 1 9 8 7 6 5 4 \\
2 1 3 8 7 9 5 4 6\\
7 9 8 {\bf 4} 6 5 1 3 2\\
9 8 7 6 5 {\bf 4} 3 2 1\\
8 7 9 5 4 6 2 1 3\\
4 6 5 1 3 2 7 9 8\\
6 5 4 3 2 1 9 8 7\\
5 4 6 2 1 3 8 7 9\\
\end{array}
~~~
r_3(L)=\begin{array}{c}
9 7 8 3 1 2 6 4 5\\
7 8 9 1 2 3 4 5 6\\
8 9 7 2 3 1 5 6 4\\
3 1 2 {\bf 6} 4 5 9 7 8\\
1 2 3 4 5 {\bf 6} 7 8 9\\
2 3 1 5 6 4 8 9 7\\
6 4 5 9 7 8 3 1 2\\
4 5 6 7 8 9 1 2 3\\
5 6 4 8 9 7 2 3 1\\
\end{array}
\]

We have checked that $r_1(L)$ is not orthogonal to $s_0(L)$ because $(4,6)$ is repeated and $r_1(L)$ is not orthogonal to $s_2(L)$ because $(4,4)$ is repeated. Similarly, $r_3(L)$ is not orthogonal to $s_0(L)$ because $(6,6)$ is repeated and $r_3(L)$ is not orthogonal to $s_2(L)$ because $(6,4)$ is repeated. These are visualized by pairing the bold face numbers in $r_1(L), s_0(L), s_2(L), r_3(L)$.

Therefore, we have $\{L, r_1(L)\}$, $\{L, s_0(L) \}$, $\{L, s_2(L)\}$, or $\{L, r_3(L)\}$ as a maximal mutually orthogonal Latin square subset of $D_8(L)$. Hence the maximum number of mutually orthogonal Latin squares in $D_8(L)$ is two.
\end{proof}

If a Latin square $A$ of order $n$ is orthogonal to $\sigma(A)$ for some $\sigma \in D_8(A)$, we call such $A$ {\em a dihedral Latin square}. We recall that a Latin square $A$ is {\em self-orthogonal} if it is orthogonal to its transpose~\cite{a5}. Since the transpose of $A$ can be represented as $s_1(A)$ ($s_1$ is a main diagonal reflection), the concept of a dihedral Latin square includes the concept of a self-orthogonal Latin square. For example, Choi's two Latin squares $L, N$ of order 9 are dihedral since $N=s_2(L)$.

Let us take another example as follows.

$$A =
\begin{matrix}
1 & 2 & 3 & 4\\
3 & 4 & 1 & 2\\
4 & 3 & 2 & 1\\
2 & 1 & 4 & 3\\
\end{matrix}
~~~
r_1(A) =
\begin{matrix}
4 & 2 & 1 & 3\\
3 & 1 & 2 & 4\\
2 & 4 & 3 & 1\\
1 & 3 & 4 & 2\\
\end{matrix}
~~~
s_1(A) =
\begin{matrix}
1 & 3 & 4 & 2\\
2 & 4 & 3 & 1\\
3 & 1 & 2 & 4\\
4 & 2 & 1 & 3\\
\end{matrix}
$$

Then $A$ and $r_1(A)$ are a pair of orthogonal Latin squares. So $A$ is a dihedral Latin square. Similarly, $A$ and $s_1(A)$ are orthogonal. So $A$ is self-orthogonal too. However $r_1(A)$ is not orthogonal to $s_1(A)$ since $(1,4)$ is repeated. By the previous theorem, the maximum number of mutually orthogonal Latin squares in the set $D_8(A)$ is 2.

Consider Choi's type Latin squares $A \oplus_S A$, $r_1(A) \oplus_S r_1(A)$, and $s_1(A) \oplus_S s_1(A)$. Then  $A \oplus_S A$ is orthogonal to both $r_1(A) \oplus_S r_1(A)$ and $s_1(A) \oplus_S s_1(A)$.


\section{Magic squares from Latin squares}
	
\begin{defi}{\em
A \textit{magic square} of order $n$ is an $n\times n$ array (or matrix) of the $n^2$ consecutive integers with the sums of each row, each column, each main diagonal, and each antidiagonal are the same.
}
\end{defi}

For example,
$$\begin{matrix}
4 & 9 & 2\\
3 & 5 & 7\\
8 & 1 & 6\\
\end{matrix}$$
is a magic square of order 3 since the sums of each row, column, main diagonal and antidiagonal are the same. Similarly, for order $n$ Latin square, we assume the symbols are $\{1,2,\cdots,n^2\}$.
	
Then the question is what the relation between Latin squares and magic squares is. We need the following definition.

\begin{defi}{\em
Let $A$ be a Latin square of order $n$. Then, $A$ is called a \textit{double-diagonal} Latin square~\cite{Hilton},~\cite{a4} if the $n$ entries in main diagonal are all distinct and the $n$ entries in antidiagonal are also all distinct.
}
\end{defi}
A construction of orthogonal double-diagonal Latin squares has been actively studied ~\cite{a3},~\cite{Brown},~\cite{a5}.

\begin{thm}{\rm (\cite{a1})}\label{thm:diag_lat}
 Suppose a pair of orthogonal double-diagonal Latin squares of order $n$ exist. Then a magic square of order $n$ can be constructed from them.
\end{thm}

\begin{defi}{\em
Suppose $A = (a_{ij})$ and $B = (b_{ij})$ are orthogonal Latin squares of order $n$. Then define an $n \times n$ square $A +_S B$ by
$$A +_S B = (n (a_{ij} - 1) + b_{ij}).$$
}
\end{defi}
	
This $A +_S B$ is not necessarily a magic square since its sums of two main diagonals is not the same as its sums of columns or rows. Theorem~\ref{thm:diag_lat} states that if the two Latin squares $A$ and $B$ are orthogonal and double-diagonal, then $A +_S  B$ is a magic square.

And the another noticeable point is that the pair of Choi's orthogonal Latin squares is not double-diagonal. However, Choi's squares also can produce a magic square even though they are not double-diagonal.

\begin{thm}\label{thm:main_anti}
If there is a pair of orthogonal Latin squares $A$ and $B$ of order $n$ such that the sum of main diagonal of each of $A$ and $B$ is $n(n+1)/2$ and the sum of antidiagonal of each of $A$ and $B$ is $n(n+1)/2$, then $A +_S B$ is a magic square of order $n$.
\end{thm}

\begin{proof} Suppose $A=(a_{ij})$ and $B=(b_{ij}), (i,j\in \{1,2,\cdots,n\})$ are orthogonal Latin squares such that
$$\sum_{i=1}^{n} a_{ii} = \sum_{i=1}^{n} b_{ii} = \frac {n(n+1)}{2}$$ \\ and \\
$$\sum_{i=1}^{n} a_{i(n+1-i)} = \sum_{i=1}^{n} b_{i(n+1-i)} = \frac {n(n+1)}{2}.$$
Now define $M=(m_{ij})$ by $M=(m_{ij})=(n(a_{ij}-1)+b_{ij})$. We want to show that $M$ is a magic square. Since $1\leq a_{ij},\, b_{ij}\leq n$ for all $(i,j)$, $1\leq n(a_{ij}-1)+b_{ij}\leq n^2$. We first show that each $m_{ij}$ is distinct. Suppose $n(a_{uv}-1)+b_{uv}=n(a_{st}-1)+b_{st}$. Then $n(a_{uv}-a_{st})=b_{st}-b_{uv}$. So $n\mid (b_{st}-b_{uv})$. However, $1\leq b_{ij}\leq n$ for all $(i,j)$, so $1-n\leq b_{st}-b_{uv}\leq n-1$. Thus $n \mid (b_{st}-b_{uv})$ implies $b_{st}-b_{uv}=0$ and so $a_{uv} = a_{st}$. Since $A$ and $B$ are orthogonal Latin squares, $(u,v)=(s,t)$. Hence if $(u,v)\ne (s,t)$ then $n(a_{uv}-a_{st})\ne b_{st}-b_{uv}$ so all $m_{ij}$ are distinct.

Now calculate the sums.
$$\sum_{i=1}^{n} \{n(a_{ij}-1)+b_{ij}\} = n\sum_{i=1}^{n} a_{ij} -n^2 + \sum_{i=1}^{n} b_{ij} = \frac {n(n^2+1)}{2},$$

$$\sum_{j=1}^{n} \{n(a_{ij}-1)+b_{ij}\} = n\sum_{j=1}^{n} a_{ij} -n^2 + \sum_{j=1}^{n} b_{ij} = \frac {n(n^2+1)}{2},$$

$$\sum_{i=1}^{n} \{ n(a_{ii}-1)+b_{ii} \} = n\sum_{i=1}^{n} a_{ii} -n^2 + \sum_{i=1}^{n} b_{ii} = \frac {n(n^2+1)}{2},$$
and similarly,
$$\sum_{i=1}^{n} \{n(a_{i(n+1-i)}-1)+b_{i(n+1-i)} \}= \frac {n(n^2+1)}{2}.$$
Thus the sums are the same. Hence $M$ is a magic square.
\end{proof}

We have an existence theorem satisfying Theorem~\ref{thm:main_anti}.

\begin{thm}\label{thm:odd}
For any odd number $n \ge 3$, there exists a pair of orthogonal Latin squares each of whose sum of main diagonal (and antidiagonal respectively) is $n(n+1)/2$.
\end{thm}

\begin{proof}
Suppose $n=2k-1$ where $k \ge 2$. Let $A_n=(a_{ij})$ be a matrix where each descending diagonal from left to right is constant like following matrix:
$$A_n \, = \;
\begin{matrix}
k & n & k-1 & n-1 & \ddots & k+2 & 2 & k+1 & 1 \\
1 & k & \ddots & \ddots & \ddots & \ddots & k+2 & 2 & k+1 \\
k+1 & \ddots & k & \ddots & \ddots & n-1 & \ddots & k+2 & 2 \\
2 & \ddots & \ddots & \ddots & n & k-1 & n-1 & \ddots & k+2 \\
\ddots & \ddots & \ddots & 1 & k & n & \ddots & \ddots & \ddots \\
k-2 & \ddots & 2 & k+1 & 1 & \ddots & \ddots & \ddots & n-1 \\
n-1 & k-2 & \ddots & 2 & \ddots & \ddots & k & \ddots & k-1 \\
k-1 & n-1 & k-2 & \ddots & \ddots & \ddots & \ddots & k & n \\
n & k-1 & n-1 & k-2 & \ddots & 2 & k+1 & 1 & k \\
\end{matrix}$$

In particular, if $n=3$ and $k=2$, we get Latin square $A_3$ in Section 2. Since Latin square $B_3$ in Section 2 is obtained by reflecting $A_3$ along the 2nd column of $A_3$, it is natural to reflect $A_n$ along the $k$th column of $A_n$ as follows.

The sum of main diagonal and the sum of antidiagonal of $A_n$ are $n(n+1)/2$ since $\sum_{i=1}^{n} a_{ii}=n\times k=n(n+1)/2$ and $\sum_{i=1}^{n} a_{i(n+1-i)}=\sum_{i=1}^{n} i = n(n+1)/2$. Recall that $s_2(A)$ is the Latin square obtained by reflecting along the middle vertical line of $A_n$. Then $s_{2}(A_n)$ has the same sum of the main diagonal (and antidiagonal respectively) of $A_n$ since the trace of $A_n$, $tr(A_n)$ is the sum of antidiagonal (and main diagonal respectively) of $s_{2}(A_n)$.

Now it remains to show that $A$ and $s_{2}(A_n)$ are orthogonal. There is a one-to-one correspondence between pandiagonals of $A_n$ and line equations; let $y=x+\alpha$ be a line with $\alpha\in \mathbb{Z}_{n}$. Then each constant pandiagonal corresponds to each equation of line. For example, $y=x$ corresponds to the diagonal constant $k$ in $A_n$ since $k=a_{ij} \Leftrightarrow i=j$ in $\mathbb{Z}_{n}$. (i,e. $(i,j)$ is a root of $y=x$ in $\mathbb{Z}_{n}$). Similarly, $x=y-2$ corresponds to the constant $k-1$, $\cdots$,  $x=y-(n-1)$ corresponds to the constant 1. And $x=y+(n-1)$ corresponds to the constant $n$, $x=y+(n-3)$ corresponds to the constant $n-1$, $\cdots$, $x=y+2$ corresponds to the constant $k+1$. Then we can do this to $s_{1}(A)$; similarly, $x=-y$ corresponds to the constant $k$, $\cdots$,  $x=-y-(n-1)$ corresponds to the constant 1. And $x=-y+(n-1)$ corresponds to the constant $n, \cdots$ , $x=-y+2$ corresponds to the constant $k+1$. Any two lines $x=y+\alpha$ and $x=-y+\beta$ have exactly one unique root. It means that an entry $(a_{ij},a_{i(n+1-j)})$ appears only once.
\end{proof}

 By the above theorem, we get a magic square constructed from a pair of orthogonal Latin squares which are not double-diagonal. Although there are many other ways to construct magic squares, our method is the way Choi obtained magic squares from two orthogonal non-double-diagonal Latin squares.

However, we can ask a question ''What does happen if $n$ is even?'' It is well known that a pair of orthogonal Latin square does not exist when $n=2$ and $n=6$, and so it is more difficult to get an even order magic square consisting of a pair of Latin squares. So we construct magic squares of some even order cases in a different way.

\begin{lem}~\label{thm:knonecker_diagonal}
 Suppose that a Latin square $A_{1}$ of order $n$ has main diagonal and antidiagonal sums $n(n+1)/2$ respectively and that a Latin square $B_{1}$ of order $m$ has main diagonal and antidiagonal sums $m(m+1)/2$ respectively. Then $A_{1}\otimes B_{1}$ is a Latin square of order $mn$ with main diagonal and antidiagonal sums $nm(nm+1)/2$ respectively.
\end{lem}

\begin{proof}
The fact that $A_{1}\otimes_S B_{1}$ is a Latin square of order $mn$ follows from Theorem 2.2. It remains to show that the two sums give $nm(nm+1)/2$.

First we consider the sum of main diagonal of $A_{1}\otimes_S B_{1}$. By definition of $A_{1}\otimes_S B_{1}$, its main diagonal sum is equal to

$$ m\{m\sum_{i=1}^n a_{ii} -mn\}+ (\sum_{i=1}^m b_{ii})n = m^2 \left \{\frac{n(n+1)}{2} -n \right \} + \frac{mn(m+1)}{2}=\frac{mn(mn+1)}{2}.$$

Similarly its antidiagonal sum is equal to
$$ m\{m\sum_{i=1}^n a_{i(n+1-i)} -mn\}+ (\sum_{i=1}^m b_{i(n+1-i)})n  = m^2 \left \{ \frac{n(n+1)}{2} -n \right \} + \frac{mn(m+1)}{2}=\frac{mn(mn+1)}{2}.$$
This completes the proof.
\end{proof}

\begin{thm}
For every $n$ with $n \equiv 2 \pmod{4}$, there exists a pair of orthogonal Latin squares each of whose sum of main diagonal (and antidiagonal, respectively) is $n(n+1)/2$.
\end{thm}

\begin{proof}
Define four Latin squares by
$$A_{1} \, = \;
\begin{matrix}
1 & 4 & 3 & 2 \\
4 & 1 & 2 & 3 \\
2 & 3 & 4 & 1 \\
3 & 2 & 1 & 4 \\
\end{matrix}
\quad A_{2} \, = \;
\begin{matrix}
3 & 1 & 2 & 4 \\
4 & 2 & 1 & 3 \\
2 & 4 & 3 & 1 \\
1 & 3 & 4 & 2 \\
\end{matrix}$$
\\
$$B_{1} \, = \;
\begin{matrix}
1 & 5 & 8 & 4 & 2 & 6 & 7 & 3 \\
3 & 8 & 5 & 2 & 4 & 7 & 6 & 1 \\
8 & 3 & 2 & 5 & 7 & 4 & 1 & 6 \\
6 & 2 & 3 & 7 & 5 & 1 & 4 & 8 \\
2 & 6 & 7 & 3 & 1 & 5 & 8 & 4 \\
4 & 7 & 6 & 1 & 3 & 8 & 5 & 2 \\
7 & 4 & 1 & 6 & 8 & 3 & 2 & 5 \\
5 & 1 & 4 & 8 & 6 & 2 & 3 & 7 \\
\end{matrix}
\quad B_{2} \, = \;
\begin{matrix}
1 & 4 & 5 & 8 & 2 & 6 & 3 & 7 \\
4 & 1 & 8 & 5 & 6 & 2 & 7 & 3 \\
3 & 8 & 1 & 6 & 5 & 7 & 2 & 4 \\
8 & 3 & 6 & 1 & 7 & 5 & 4 & 2 \\
7 & 5 & 4 & 2 & 8 & 3 & 6 & 1 \\	
5 & 7 & 2 & 4 & 3 & 8 & 1 & 6 \\
6 & 2 & 7 & 3 & 4 & 1 & 8 & 5 \\
2 & 6 & 3 & 7 & 1 & 4 & 5 & 8 \\
\end{matrix}$$

Then $A_{1}$ and $A_{2}$ are orthogonal. $B_{1}$ and $B_{2}$ are also orthogonal. So we can construct two orthogonal Latin squares of order $4k$ (where $k$ is an integer) using the following way.

If we want to construct of orthogonal Latin squares of order $4t$ with $t$ odd, then we can make two Latin squares $A_1\otimes_S C_1$ and $A_2\otimes_S C_2$ where $C_1$ and $C_2$ are orthogonal Latin squares of order $t$ and the sums of diagonal and antidiagonal are $t(t+1)/2$ (By Theorem~\ref{thm:odd}, we can get such pair of Latin squares). Then $A_1\otimes_S C_1$ and $A_2\otimes_S C_2$ are orthogonal by Theorem 2.2 and each sum of their diagonal and antidiagonal is $4t(4t+1)/2$ by Lemma 4.4.

Or if we want to construct orthogonal Latin squares of order $2^p$ with $p\geq3$, we recursively use the substituted Kronecker products of  $A_1, B_1, A_2,$ and $B_2$.

So we can construct orthogonal Latin Squares of an even order $n$ which is not of the form of $2r$ ($r$ is odd) each of whose sum of diagonal (and antidiagonal, respectively) is $n(n+1)/2$.
\end{proof}

\begin{cor}
For any integer $n$ with $n\neq2r$ where $r$ is odd, there exists a pair of non-double-diagonal orthogonal Latin Squares of order $n$ such that the pair of Latin squares can produce a magic square of order $n$.
\end{cor}

\begin{proof}
By Theorems 4.2, 4.3, and 4.5, we can construct a magic square of order $n$ where $n\neq2r$ ($r$ is odd).
\end{proof}

Therefore, Choi's orthogonal Latin squares of various orders give a new way to construct magic squares based on non-double-diagonal orthogonal Latin squares.

\end{document}